\documentclass[reqno,12pt]{article}
\usepackage[margin=1.0 in  ]{geometry}         
\usepackage{graphicx}
\usepackage{stmaryrd}
\usepackage{framed}
\usepackage{amssymb}
\usepackage{epstopdf}
\usepackage{amsmath,amsfonts,amssymb}
\usepackage{enumerate}
\usepackage{multirow}
\usepackage{fixltx2e}
\usepackage{mathpazo}

\usepackage{xcolor}
\usepackage{soul}

\usepackage{theorem}

\newcommand{\scal}[2]{\left\langle{#1},{#2}  \right\rangle}
\newcommand{\menge}[2]{\big\{{#1}~\big |~{#2}\big\}}
\newcommand{\nnn}{\ensuremath{{n\in{\mathbb N}}}}
\newcommand{\Fix}{\ensuremath{\operatorname{Fix}}}
\newcommand{\Id}{\ensuremath{\operatorname{Id}}}
\newcommand{\ran}{\ensuremath{\operatorname{ran}}}

\newtheorem{theorem}{Theorem}[section]
\newtheorem{lemma}[theorem]{Lemma}

\theoremstyle{plain}{\theorembodyfont{\rmfamily}
}
\theoremstyle{plain}{\theorembodyfont{\rmfamily}
}
\theoremstyle{plain}{\theorembodyfont{\rmfamily}
\newtheorem{algorithm}[theorem]{Algorithm}}
\theoremstyle{plain}{\theorembodyfont{\rmfamily}
\newtheorem{example}[theorem]{Example}}

\theoremstyle{plain}{\theorembodyfont{\rmfamily}
\newtheorem{remark}[theorem]{Remark}}

\providecommand{\bx}{\mathbf{x}}
\providecommand{\be}{\mathbf{e}}

\providecommand{\RR}{\mathbb{R}}

\newcommand{\sepp}{\setlength{\itemsep}{-3pt}}

\title{A projection method for approximating
fixed points of quasi nonexpansive mappings without
the usual demiclosedness condition }

\author{Heinz H. Bauschke\thanks{Mathematics. Irving K. Barber School, University of British Columbia,  Kelowna, B.C. V1V 1V7, Canada. {\tt heinz.bauschke@ubc.ca}. },
Jiawei Chen\thanks{School of Mathematics and Statistics, Wuhan University,
Wuhan 430072, P.R.~China. {\tt J.W.Chen713@163.com}. }, and
Xianfu Wang\thanks{Mathematics. Irving K. Barber School, University of
British Columbia,  Kelowna, B.C. V1V 1V7, Canada. {\tt shawn.wang@ubc.ca}.
}}

\date{November 7, 2012}

\begin{document}

\maketitle

\begin{abstract}
\noindent
We introduce and analyze an abstract algorithm that aims to find
the projection onto a closed convex subset of a Hilbert space.
When specialized to the fixed point set of a
quasi nonexpansive mapping, the
required sufficient condition (termed ``fixed-point closed'')
is less restrictive
than the usual conditions based on the demiclosedness principle.
A concrete example of a subgradient projector is presented which
illustrates the applicability of this generalization.
\end{abstract}


\section{Introduction}

Throughout this note, we assume that
\begin{equation}
\text{$X$ is a real Hilbert space with inner product $\scal{\cdot}{\cdot}$
and norm $\|\cdot\|$.}
\end{equation}
Suppose that
\begin{equation}
\label{e:global1}
\text{$C$ is a closed convex subset of $X$, and $x_0\in X$.}
\end{equation}
We are interested in finding the projection (nearest point mapping)
$P_Cx_0$, i.e., the unique solution to the optimization
problem
\begin{equation}
\label{e:theproblem}
d(x_0,C) := \min_{c\in C}\|x_0-c\|,
\end{equation}
especially when $C$ is the fixed point set of some operator $T\colon X\to
X$. 
It will be convenient to set,
for arbitrary given vectors $x$ and $y$ in $X$,
\begin{equation}
H(x,y) := \menge{z \in X}{\|y-z\|\leq\|x-z\|} = \menge{z\in
X}{2\scal{z}{x-y}\leq\|x\|^2-\|y\|^2}.
\end{equation}
Note that $H(x,y)$ is equal to either $X$ (if $x=y$) or a halfspace;
in any case, the projection onto $H(x,y)$ is easy to compute and has a well
known closed form.
In order to solve \eqref{e:theproblem}, we shall study the following simple
abstract iteration:

\begin{algorithm}
\label{a:1}
Recall the assumption~\eqref{e:global1}, and set $C_0 = X$.
Given $\nnn$ and $x_n\in X$, pick $y_n\in X$, and set
\begin{equation}
C_{n+1} := C_n \cap H(x_n,y_n)
\;\;\text{and}\;\;
x_{n+1} = P_{C_{n+1}}x_0.
\end{equation}
\end{algorithm}
Observe that if the sequence is well defined, then
\begin{equation}
\label{e:nested}
C_0\supseteq C_1 \supseteq \cdots C_n \supseteq C_{n+1}\supseteq \cdots
\end{equation}
and so
\begin{equation}
\label{e:increases}
\|x_0-x_n\| = d(x_0,C_n) \leq d(x_0,C_{n+1})=\|x_0-x_{n+1}\|
\end{equation}
for every $\nnn$.
It then follows that
\begin{equation}
\label{e:beta}
\beta := \lim_\nnn\|x_0-x_n\| = \sup_\nnn\|x_0-x_n\| \in
\left[0,+\infty\right]
\end{equation}
is well defined.
Furthermore, if $m < n$, then
$x_n\in C_m$ which implies
\begin{equation}
\label{e:kolmo}
\scal{x_n-x_m}{x_0-x_m}\leq 0
\end{equation}
as well as
\begin{equation}
\label{e:bla}
\|y_m-x_n\|\leq\|x_m-x_n\|
\end{equation}
because
$x_n \in C_n \subseteq C_{m+1}\subseteq H(x_m,y_m)$.

\begin{lemma}
\label{l:key}
Suppose that the sequence $(x_n)_\nnn$ is generated by
Algorithm~\ref{a:1}.
Suppose also that
for every subsequence $(x_{k_n})_\nnn$ of $(x_n)$, we have
\begin{equation}
\label{e:superdemi}
\left.
\begin{array}{c}
x_{k_n}\to \bar{x}\\
x_{k_n}-y_{k_n}\to 0
\end{array}
\right\}
\;\;\Rightarrow\;\;
\bar{x}\in C.
\end{equation}
Then every bounded subsequence of $(x_n)_\nnn$ must converge to a point in
$C$.
\end{lemma}
\begin{proof}
Let $(x_{k_n})_\nnn$ be a bounded subsequence of $(x_n)_\nnn$.
It follows from \eqref{e:increases} that $\beta<+\infty$.
Let $n> m$.
Using \eqref{e:kolmo}, we obtain
\begin{subequations}
\begin{align}
\|x_{k_n}-x_{k_m}\|^2 &= \|x_{k_n}-x_0\|^2 - \|x_{k_m}-x_0\|^2 +
2\scal{x_{k_n}-x_{k_m}}{x_0-x_{k_m}}\\
&\leq \|x_{k_n}-x_0\|^2 - \|x_{k_m}-x_0\|^2\\
&\to \beta^2-\beta^2 = 0\;\;\text{as $n\geq m\to +\infty$.}
\end{align}
\end{subequations}
Hence $(x_{k_n})_\nnn$ is a Cauchy sequence.
Thus, there exists $\bar{x}\in X$ such that $x_{k_n}\to \bar{x}$.
Now, from \eqref{e:bla}, we obtain
$\|y_{k_n}-x_{k_{n+1}}\| \leq \|x_{k_n}-x_{k_{n+1}}\| \to
\|\bar{x}-\bar{x}\| = 0$ and thus $y_{k_n}-x_{k_{n+1}}\to 0$.
It follows that $x_{k_n}-y_{k_n} = (x_{k_n}-x_{k_{n+1}}) + (x_{k_{n+1}} -
y_{k_n})\to 0$.
Now apply \eqref{e:superdemi}.
\end{proof}

The previous result allows us to derive the following dichotomy result.

\begin{theorem}[dichotomy]
\label{t:main}
Suppose that $(x_n)_\nnn$ is generated by
Algorithm~\ref{a:1}, that $(\forall\nnn)$ $C\subseteq C_n$,
and that for every subsequence $(x_{k_n})_\nnn$ of $(x_n)$, we have
\begin{equation}
\left.
\begin{array}{c}
x_{k_n}\to \bar{x}\\
x_{k_n}-y_{k_n}\to 0
\end{array}
\right\}
\;\;\Rightarrow\;\;
\bar{x}\in C.
\end{equation}
Then exactly one of the following holds:
\begin{enumerate}
\item
\label{t:main1}
$C\neq\varnothing$ and $x_n\to P_Cx_0$.
\item
\label{t:main2}
$C=\varnothing$ and $\|x_n\|\to+\infty$.
\end{enumerate}
\end{theorem}
\begin{proof}
Note that
\begin{equation}
\label{e:rot}
(\forall\nnn)\;\;
\|x_0-x_{n}\| = d(x_0,C_n)\leq d(x_0,C).
\end{equation}

\ref{t:main1}:
Assume that $C\neq\varnothing$.
Then $(x_n)_\nnn$ is bounded by \eqref{e:rot}.
By Lemma~\ref{l:key}, $\bar{x} := \lim_\nnn x_n \in C$.
In turn, \eqref{e:rot} yields $\|x_0-\bar{x}\|\leq d(x_0,C)$.
Therefore, $\bar{x}=P_Cx_0$, as claimed.

\ref{t:main2}:
Suppose that $\|x_n\|\not\to+\infty$.
Then $(x_n)_\nnn$ contains a bounded subsequence which,
by Lemma~\ref{l:key}, must converge to a point in $C$.
Hence if $C=\varnothing$, then $\|x_n\|\to+\infty$.
\end{proof}

\begin{remark}
Several comments regarding Theorem~\ref{t:main} are in order.
\begin{enumerate}
\item
Algorithm~\ref{a:1} is related to a method
studied by Takahashi et al in \cite[Theorem~4.1]{TTT}.
(See also \cite[Theorem~2]{RS2,RS} for Bregman-distance based variants.)
While that method is more flexible in some ways, our method has the
advantage of requiring neither nonexpansiveness of the given operator
nor the nonemptiness of the target set.
\item
Our proofs are different because we establish strong convergence
directly via a Cauchy sequence argument. The proofs mentioned in
the previous item are based on a Kadec-Klee property or on Opial's
property.
(We expect that our proof will generalize to Bregman distances,
possibly incorporating errors and families of operators.)
\item
As we shall see in Section~\ref{s:sp} below,
our framework encompasses subgradient
projectors which are important in optimization.
\item
The computation of the sequence $(x_n)_\nnn$ requires
to compute projections of the \emph{same} initial point $x_0$
onto polyhedra (intersections of finitely many halfspaces).
While this is not necessarily an easy task, this is considered to be
a standard quadratic programming problem in convex optimization.
Moreover, since $C_{n+1}$ is constructed from $C_n$ by intersecting
with the halfspace $H(x_n,y_n)$, it seems plausible to apply
\emph{active set methods} (with a warm start) to solve these
projections. While a detailed excursion on this matter is beyond the
scope of this paper, we do refer the reader to \cite{Arioli,Huynh,Nurminski}
for references
on computing projections onto polyhedra.
\end{enumerate}
\end{remark}

\section{An application to finding nearest fixed points}

Recall that $T\colon X\to X$ is called
\emph{nonexpansive} if
\begin{equation}
(\forall x\in X)(\forall y\in X)\quad
\|Tx-Ty\|\leq\|x-y\|;
\end{equation}
moreover, $T$ is \emph{quasi nonexpansive} if
\begin{equation}
(\forall x\in X)(\forall y\in \Fix T)\quad
\|Tx-y\|\leq\|x-y\|,
\end{equation}
where $\Fix T := \menge{x\in X}{x=Tx}$.
See \cite{GK,GR,BC} for further information on the fixed point theory
of nonexpansive mappings.

The next result is readily checked.

\begin{lemma}
\label{l:easy}
Let $T\colon X\to X$ be quasi nonexpansive.
Consider the following properties:
\begin{enumerate}
\item
\label{l:easy1}
$T$ is nonexpansive.
\item
\label{l:easy2}
$T$ is continuous.
\item
\label{l:easy3}
$T$ is \emph{fixed-point closed}, i.e.,
if
$x_n\to\bar{x}$ and
$x_n-Tx_n\to 0$, then
$\bar{x}\in\Fix T$.
\end{enumerate}
Then \ref{l:easy1}$\Rightarrow$\ref{l:easy2}$\Rightarrow$\ref{l:easy3}.
\end{lemma}

\begin{remark}
It is well known that if $T\colon X\to X$ is nonexpansive,
then
\begin{equation}
\left.
\begin{array}{c}
x_n\rightharpoonup \bar{x}\\
x_n-Tx_n \to 0
\end{array}
\right\}
\;\;\Rightarrow\;\;
\bar{x}\in\Fix T;
\end{equation}
this is the famous demiclosedness principle --- to be precise,
this states that $\Id-T$ is demiclosed at $0$.
For recent results on this principle, see \cite{Demi} and the references
therein.
It is clear that demiclosedness of $\Id-T$ at $0$ implies
that $T$ is fixed-point closed; the converse, however, is false
(see Example~\ref{ex:bad} below).
\end{remark}

Our main result now yields easily the following result, which
by Lemma~\ref{l:easy} is applicable in particular when $T$ is nonexpansive.
(See also \cite[Theorem~4.1]{TTT} for extensions in the nonexpansive case.)

\begin{theorem}[trichotomy]
\label{t:fix}
Let $T\colon X\to X$ be quasi nonexpansive and fixed-point closed,
let $x_0\in X$, and set $C_0 := X$.
Given $\nnn$ and $x_n$, set
\begin{equation}
C_{n+1} := C_n \cap H(x_n,Tx_n)
\;\;\text{and}\;\;
x_{n+1} = P_{C_{n+1}}x_0.
\end{equation}
Then exactly one of the following holds:
\begin{enumerate}
\item
\label{t:fix1}
$\Fix T\neq\varnothing$ and $x_n\to P_{\Fix T}x_0$.
\item
\label{t:fix2}
$\Fix T = \varnothing$ and $\|x_n\|\to+\infty$.
\item
\label{t:fix3}
$\Fix T = \varnothing$ and the sequence is not well defined (i.e.,
$C_{n+1}$ is empty for some $n$).
\end{enumerate}
\end{theorem}
\begin{proof}
Set $C = \Fix T$, and $(y_n)_\nnn = (Tx_n)_\nnn$ provided that $(x_n)_\nnn$
is well defined.
In this case, it is clear that \eqref{e:superdemi} holds because $T$ is
fixed-point closed.

\ref{t:fix1}: Assume that $C\neq\varnothing$.
If $C_n\neq\varnothing$ and $C\subseteq C_n$, then
$(\forall c\in C)$
$\|Tx_n-c\|\leq \|x_n-c\|$ and so $c\in H(x_n,Tx_n)$.
It follows that $C\subseteq C_{n+1}$ and the sequence $(x_n)_\nnn$
is well defined.
The conclusion thus follows from Theorem~\ref{t:main}.

\ref{t:fix2}\&\ref{t:fix3}: Assume that
$C=\varnothing$.
If $(x_n)_\nnn$ is not well defined, then \ref{t:fix3} happens.
Finally, if $(x_n)_\nnn$ is well defined, then \ref{t:fix2} occurs
again by Theorem~\ref{t:main}.
\end{proof}

Let us now illustrate the three alternatives in Theorem~\ref{t:fix}.

\begin{example}
Suppose that $X=\RR$ and set $T := \alpha\Id$, where
$\alpha\in\left[0,1\right[$.
Then $T$ is nonexpansive with $\Fix T = \{0\}$.
Let $x_0\geq 0$. Then $Tx_0 = \alpha x_0$ and
$C_1 = \left]-\infty,(\alpha+1)/2x_0\right]$.
Thus, $x_1 = (\alpha+1)/2x_0$.
It follows inductively that $(x_n)_\nnn$ is well defined and
\begin{equation}
(\forall \nnn)\quad
x_n = \big((\alpha+1)/2\big)^nx_0 \to 0 = P_{\Fix T}x_0,
\end{equation}
as is also guaranteed by Theorem~\ref{t:fix}\ref{t:fix1}.
\end{example}

\begin{example}
Suppose that $X=\RR$ and set $T\colon X\to X\colon x\mapsto x+\alpha$,
where $\alpha>0$.
Clearly, $T$ is nonexpansive and $\Fix T =\varnothing$.
One checks that $x_n = x_0+n\alpha/2$;
hence, $|x_n| \to +\infty$.
\end{example}

\begin{example}
Suppose $X=\RR$, let $\sigma\colon X\to\{-1,+1\}$, and set
$T_\sigma \colon X\mapsto X\colon x\mapsto x+\sigma(x)$.
For trivial reasons, $T_\sigma$ is quasi nonexpansive (since $\Fix T_\sigma =
\varnothing$) and $T_\sigma$ is fixed-point closed (since
$\ran(\Id-T_\sigma)\subseteq \{+1,-1\}$).
We now assume that $\sigma(0)=1$ and $\sigma(1/2)=-1$.
Let $x_0=0$. Then $C_1 = \left[1/2,+\infty\right[$,
$x_1 = 1/2$ and $C_2 = C_1 \cap \left]-\infty,0\right] =
\varnothing$, which means the algorithm terminates.
\end{example}

\section{Subgradient projector}
\label{s:sp}

The astute reader will ask whether the fairly general assumptions on $T$ in
Theorem~\ref{t:fix}, i.e., that
``$T$ be quasi nonexpansive and fixed-point closed'',
are really needed in applications.
In this section, we provide an example that not only requires this
generality but that also does not satisfy the usual demiclosedness type
assumptions seen in this area.

To this end,
let
\begin{equation}
f\colon X\to\RR
\end{equation}
be convex, continuous, and G\^ateaux differentiable such that
$f\geq 0$
and
\begin{equation}
C := \menge{x\in X}{f(x)\leq 0} = \{0\}.
\end{equation}
Write $g := \nabla f$ for convenience.
The \emph{subgradient projector} in this case is defined by
\begin{equation}
T\colon X\to X\colon
x\mapsto \begin{cases}
x, &\text{if $x=0$;}\\
x - \frac{f(x)}{\|g(x)\|^2}g(x), &\text{if $x\neq 0$.}
\end{cases}
\end{equation}
Then it follows (from e.g., \cite[Proposition~2.3]{MOR}) that
$T$ is \emph{quasi firmly nonexpansive}, i.e.,
\begin{equation}
\label{e:qfne}
(\forall x\in X)(\forall y\in \Fix T)\quad
\|Tx-y\|^2 + \|x-Tx\|^2 \leq \|x-y\|^2.
\end{equation}

\begin{lemma}
\label{l:sp}
\
The following hold:
\begin{enumerate}
\item
\label{l:sp0}
$T$ is quasi nonexpansive.
\item
\label{l:sp1}
$T$ is fixed-point closed.
\item
\label{l:sp2}
$T$ is continuous at $0$.
\item
\label{l:sp3}
If $f$ is Fr\'echet differentiable, then $T$ is continuous.
\end{enumerate}
\end{lemma}
\begin{proof}
\ref{l:sp0}: This follows immediately from \eqref{e:qfne}.

\ref{l:sp1}:
Let $(x_n)_\nnn$ be a sequence in $X$ such that $x_n\to \bar{x}$
and $x_n-Tx_n\to 0$.
We assume that $\bar{x}\neq 0$ (for if $\bar{x}=0$, then the conclusion is
trivially true) and that $(x_n)_\nnn$ lies in $X\smallsetminus\{0\}$.
To reach the required contradiction, observe first
that the continuity of $f$ yields
$f(x_n)\to f(\bar{x})> 0$.
Now $x_n-Tx_n\to 0$
$\Leftrightarrow$
$\|x_n-Tx_n\|\to 0$
$\Leftrightarrow$
$f(x_n)/g(x_n)\to 0$; thus,
\begin{equation}
\label{e:blowup}
\lim_\nnn \|g(x_n)\| = +\infty.
\end{equation}
On the other hand, $g$ is strong-to-weak continuous
(see, e.g., \cite[Proposition~17.31]{BC});
therefore, the sequence
$(g(x_n))_\nnn$ converges weakly to $g(\bar{x})$.
In particular, $(g(x_n))_\nnn$ is bounded --- but this contradicts
\eqref{e:blowup}.

\ref{l:sp2}:
Convexity yields
$(\forall x\in X\smallsetminus\{0\})$
$\scal{0-x}{\nabla f(x)} \leq f(0)-f(x)$, which implies
$f(x) \leq \scal{x}{g(x)}\leq \|x\|\|g(x)\|$; thus,
$f(x)/\|g(x)\| \leq \|x\|$.
Hence $\lim_{x\to 0} Tx = 0 = T0$, as claimed.

\ref{l:sp3}:
If $f$ is Fr\'echet differentiable, then $g$ is strong-to-strong
continuous (see, e.g., \cite[Proposition~17.32]{BC}),
which in turn yields the continuity of $T$
on $\menge{x\in X}{g(x)\neq 0} = X\smallsetminus\{0\}$.
\end{proof}

Note that Lemma~\ref{l:sp} guarantees the applicability of
Theorem~\ref{t:fix} to the subgradient projector $T$.

\begin{example}
\label{ex:bad}
Suppose that $X = \ell^2 = \menge{\bx = (x_n)_{n\geq 1}}{\sum_{n\geq
1}|x_n|^2 < +\infty}$ and set
\begin{equation}
f\colon X\to\RR\colon \bx = (x_n)_{n\geq 1} \mapsto\sum_{n\geq 1}
nx_n^{2n}.
\end{equation}
Then $f$ is well defined, convex, and continuous
(see \cite[Example~7.11]{SIREV}).
Moreover, $f$ is G\^ateaux differentiable with
$g(\bx) = \nabla f(\bx) = (2n^2x_n^{2n-1})_{n\geq 1}$.
Denote the sequence of standard unit vectors by $(\be_n)_{n\geq 1}$,
and set
\begin{equation}
(\forall n\geq 1) \quad \bx_n := \be_1 + \be_n \rightharpoonup \be_1
\end{equation}
For $n\geq 2$, we have
$f(\bx_n) = 1+n$,
$g(\bx_n) = 2\be_1 + 2n^2\be_n$;
hence $\|g(\bx_n)\| = \sqrt{4+ 4n^4}$ and thus
$f(\bx_n)/\|g(\bx_n)\| \to 0$.
It follows that
$\bx_n - T(\bx_n)\to 0$.
Since
\begin{equation}
\left.
\begin{array}{c}
\bx_n\rightharpoonup{\be_1}\\
\bx_n-T(\bx_n)\to 0
\end{array}
\right\}
\;\;\not\Rightarrow\;\;
{\be_1} = 0,
\end{equation}
we see that $\Id-T$ is \emph{not} demiclosed at $0$ and
that $T$ is not weak-to-weak continuous
however,
$T$ is fixed-point closed by  Lemma~\ref{l:sp}\ref{l:sp1}.
\end{example}

\begin{remark}
Some comments regarding Example~\ref{ex:bad} are in order.
\begin{enumerate}
\item
This example illustrates that some of the sufficient conditions
demi-closedness type conditions
provided in the literature (see, e.g.,
\cite[Proposition~2.2]{PLCSICON}) to guarantee convergence are actually
not applicable to the subgradient projector $T$ of the function $f$
defined in Example~\ref{ex:bad}. However, Theorem~\ref{t:fix} is applicable
with $T$ because of Lemma~\ref{l:sp}.
\item
Some additional work (which we omit here)
shows that $f$ is actually Fr\'echet differentiable
on $X$. Thus, by Lemma~\ref{l:sp}\ref{l:sp3}, $T$ is actually
strong-to-strong continuous.
\item
It also follows from the classical demiclosedness principle that $T$
is not nonexpansive.
\end{enumerate}
\end{remark}

\section*{Acknowledgments}
This research was carried out during a visit of JC in Kelowna in Fall
2012.
HHB was partially supported by the Natural Sciences and
Engineering Research Council of Canada and by the Canada Research Chair
Program.
JC was partially supported by the Academic Award for Excellent Ph.D.~
Candidates Funded by Wuhan University, the Fundamental Research
Fund for the Central Universities,
and the Ph.D.~short-time mobility program by Wuhan University.
XW was partially supported by the Natural
Sciences and Engineering Research Council of Canada.

\bibliographystyle{plain}

\end{document}